\documentclass[12pt]{amsart}
\usepackage{fullpage}
\usepackage{amsmath, amsthm, amssymb, graphicx, here}

\newcounter{citedtheorems}

\newtheorem{defn}{Definition}[section]
\newtheorem{theorem}[defn]{Theorem}
\newtheorem*{th-ref}{Theorem} 
\newtheorem{thm-lit}[citedtheorems]{Theorem} 
\newtheorem{cor-lit}[citedtheorems]{Corollary}

\newtheorem{cor}[defn]{Corollary}
\newtheorem{propn}[defn]{Proposition}
\newtheorem{concl}[defn]{Conclusion}
\newtheorem{conv}[defn]{Convention}
\newtheorem{claim}[defn]{Claim}
\newtheorem{lemma}[defn]{Lemma}
\newtheorem{obs}[defn]{Observation}
\newtheorem{rmk}[defn]{Remark}

\newtheorem{expl}[defn]{Example}

\newcommand{\br}{\vspace{2mm}}

\newcommand{\cols}{\operatorname{Col}}

\newcommand{\vrt}{\rule{0pt}{12pt}}
\newcommand{\svrt}{\rule{0pt}{10pt}}

\newcommand{\fss}{{\mathcal{P}}_{\aleph_0}}

\newcommand{\vp}{\varphi}

\newcommand{\ds}{-}

\def\dsp{\def\baselinestretch{1.0}\large\normalsize}
\def\ssp{\def\baselinestretch{1.0}\large\normalsize}

\begin{document}

\title{Edge distribution and density \\ in the characteristic sequence}
\author{M. E. Malliaris}
\address{Department of Mathematics, University of Chicago, 5734 S. University Avenue, Chicago, IL 60637}
\email{mem@math.uchicago.edu}

\begin{abstract}
The characteristic sequence of hypergraphs $\langle P_n : n<\omega \rangle$ associated to a formula $\vp(x;y)$,
introduced in \cite{mm-article-2}, is defined by $P_n(y_1,\dots y_n) = (\exists x) \bigwedge_{i\leq n} \vp(x;y_i)$. 
This paper continues the study of characteristic sequences,  
showing that graph-theoretic techniques, notably Szemer\'edi's celebrated regularity lemma, can
be naturally applied to the study of model-theoretic complexity via the characteristic sequence. 
Specifically, we relate classification-theoretic properties of $\vp$ and of the $P_n$ (considered as formulas) to
density between components in Szemer\'edi-regular decompositions of graphs in the characteristic sequence. In addition, we 
use Szemer\'edi regularity to calibrate model-theoretic notions of independence by describing the 
depth of independence of a constellation of sets and showing that certain failures of depth imply 
Shelah's strong order property $SOP_3$; this sheds light on the interplay of independence
and order in unstable theories. 
\end{abstract}

\maketitle

\section{Introduction}

The characteristic sequence $\langle P_n : n<\omega \rangle$ is a tool for studying the combinatorial complexity of a given formula $\vp$, 
Definition \ref{characteristic-sequence} below. It follows from \cite{mm-article-1}, \cite{mm-article-2} that 
the Keisler order \cite{keisler} localizes to the study of $\vp$-types and specifically of characteristic sequences. 
However, this article will not focus on ultrapowers.

The analysis of \cite{mm-article-2} established that characteristic sequences are essentially trivial when
the ambient theory $T$ is $NIP$, Theorem \ref{char-stable} below. In this article, we turn to the study of
characteristic sequences in the presence of the independence property. 
The framework of characteristic sequences allows us to bring a deep collection of graph-theoretic
structure theorems to bear on our investigations. 
Notably, the classic model-theoretic move of polarizing complex structure into rigid and random components
(e.g. Shelah's isolation of the independence property and the strict order property in unstable theories)
is accomplished here by the application of Szemer\'edi's Regularity Lemma, \S \ref{section:regularity} Theorem \ref{szr} below.
Because the Regularity Lemma describes a possible decomposition of \emph{any} sufficiently large graph, 
it can be applied here to understand how arbitrarily large subsets of $P_1$ generically interrelate. 

In Sections \ref{section:counting}-\ref{section:order-genericity},
we investigate how classic properties of $T$ affect the density $\delta$ attained 
between arbitrarily large $\epsilon$-regular subsets $A,B \subset P_1$ (after localization) in the sense of Szemer\'edi regularity,
where the edge relation is given by $P_2$.
The picture we obtain is as follows. When $\vp$ is stable, by Theorem \ref{char-stable}, the density (after localization) is always 1. 
When $\vp$ is simple unstable, after localization, there will be an infinite number of missing edges
but we can say something strong about their distribution: $(*)$ the density between arbitrarily large $\epsilon$-regular pairs must tend
towards $0$ or $1$ as the graphs grow (indeed, here simplicity is sufficient but not necessary). In the simple unstable case, 
a finer function counting the number of edges omitted over finite subgraphs of size $n$
is meaningful, and we give a preliminary description of its possible values in Proposition \ref{counting-thm}. 
In Section \ref{section:order-genericity}, we use model theory to relate the property $(*)$ of having arbitrarily large $\epsilon$-regular subsets 
of $P_1$ with edge density bounded away from 0 and 1 to the phenomenon of instability in the characteristic sequence, which
is strictly more complex than failure of simplicity.
In Section \ref{section:two-order-properties} we refine this phenomenon by 
defining and investigating the compatible and empty order properties. 
On the level of theories, the compatible order property characterizes the model-theoretic rigidity property $SOP_3$,
which is known to imply maximality in the Keisler order by \cite{Sh500}.

In the other direction, in Section \ref{section:randomness} we use Szemer\'edi regularity 
to bring to light a subtle model-theoretic failure of randomness,
by considering the ``depth of independence'' of a constellation of infinite sets. 
In the language of Definition \ref{constellations}, 
we show that theories which are $I^{n+1}_n$ but not $I^{n+1}_{n+1}$ for some $n>2$, are $SOP_3$.
This is a result about the fine structure of the classic SOP/IP distinction,
illustrating the tradeoff between a weaker notion of strict order ($SOP_3$) 
and a stronger notion of independence ($I^{n+1}_{n+1}$) in unstable theories.

\br

\subsection*{Acknowledgements} 
Thanks are due to my advisor Thomas Scanlon, and to Leo Harrington, for many stimulating conversations, as well as to
Scanlon's NSF grant for funding a trip to the ICM in Madrid where I first learned of Szemer\'edi's work. 
Thanks also to Laci Babai for a copy of the helpful survey \cite{rsurvey}.

\section{Preliminaries} \label{section:intro}

The following conventions will be in place throughout the article.

\begin{conv} \label{notation} \emph{(Conventions)}
\begin{enumerate}
\item If a variable or a tuple is written $x$ or $a$ rather than $\overline{x}, \overline{a}$, this does not necessarily
imply that $\ell(x), \ell(a)=1$. 
\item Unless otherwise stated, $T$ is a complete theory in the language $\mathcal{L}$.
\item A set is \emph{$k$-consistent} if every $k$-element subset is consistent, and it is \emph{$k$-inconsistent} if every $k$-element
subset is inconsistent. 
\item $\vp_\ell(x;y_1,\dots y_\ell) := \bigwedge_{i\leq l} \vp(x;y_i)$
\item $S_{\aleph_0}(\omega)$ is the set of all finite subsets of $\omega$.
\item $\epsilon, \delta$ are real numbers, with $0 < \epsilon < 1$ and $0 \leq \delta \leq 1$.
\item Let $G$ be a symmetric binary graph. We present graphs model-theoretically, i.e. as sets of vertices
on which certain edge relations hold. Throughout this article $R(x,y)$ is a binary edge relation,
which will sometimes (we will clearly say when) be interpreted as $P_2$.
\item A graph is a simple graph: no loops and no multiple edges. Definition \ref{characteristic-sequence} below implies that
$\forall x (P_1(x) \rightarrow P_2(x,x))$, but
we will, by convention, not count loops when taking $P_2$ as $R$.  
\item Given a graph $G$, with symmetric binary edge relation $R(x,y)$:
\begin{itemize} 
\item $|G|$ is the size of $G$, i.e. the number of vertices.
\item $e(G)$ is the number of edges of $G$.
\item $\hat{e}(G)$ is the number of edges omitted in $G$.
\item An \emph{empty graph} is a graph with no edges.
\item A \emph{complete graph} is a graph with all edges, i.e. in which $x,y \in G, x\neq y \implies R(x,y)$.
\item The \emph{degree} of a vertex is the number of edges which contain it. 
\item The \emph{dual graph} $G^\prime$ has the same vertices and inverted edges, i.e. 
for $x \neq y$, $G^\prime \models R(x,y)$ $\iff$ $G \models \neg R(x,y)$.  
\end{itemize}
\item Write $(X,Y)$ to indicate a a bipartite graph. Then:
\begin{itemize} 
\item $e(X,Y)$ is the number of edges between elements $x \in X$ and $y \in Y$. Note that if $G = A \cup B$
then possibly $e(G) \neq e(A,B)$, as the latter counts only edges between $A$ and $B$.
\item $\hat{e}(X,Y)$ is the number of edges omitted between elements $x \in X$ and $y \in Y$.
\item The \emph{density} of a finite bipartite graph $(X,Y)$ is $\delta(X,Y) := e(X,Y)/|X||Y|$
when $|X|, |Y| \neq 0$, and $0$ otherwise. 
\item An \emph{empty pair} is a pair of vertices $x,y$ with $\neg R(x,y)$.
\item An \emph{infinite empty pair} is $(X,Y)$ such that $|X| = |Y| \geq \aleph_0$ and for all
$x \in X$, $y \in Y$, we have $\neg R(x,y)$. 
\item A \emph{complete bipartite graph} is $(X,Y)$ such that for all $x \in X, y \in Y$, $R(x,y)$.
\item The dual $(X,Y)^\prime$ of a bipartite graph inverts precisely the edges between the components $X$ and $Y$.
\end{itemize}
\end{enumerate}
\end{conv} 

We will make extensive use of the important classification-theoretic dividing lines of stability,
simplicity, the independence property, and the strict order property; see, for instance, 
\cite{Sh:c}, Chapter II, sections 2-4 and \cite{Sh500}. A theory or a formula is NIP, also called dependent,
if it does not have the independence property; see, for instance, \cite{Usv}.

We now turn to definitions.
The characteristic sequence of hypergraphs was introduced in \cite{mm-article-2} as a tool for studying the complexity of a given formula $\vp$.
Let us set the stage by briefly reviewing some of the results obtained there.

\begin{defn} \emph{(Characteristic sequences)} \label{characteristic-sequence} 
Let $T$ be a first-order theory and $\vp$ a formula of the language of $T$. 
\begin{itemize}
\item For $n<\omega$, $P_n(z_1,\dots z_n) :=  \exists x \bigwedge_{i\leq n} \vp(x;z_i) $.
\item The \emph{characteristic sequence} of $\vp$ in $T$ is $\langle P_n : n<\omega \rangle$.
\item Write $(T,\vp) \mapsto \langle P_n \rangle$ for this association. 
\item We assume that $T \vdash \forall y \exists z \forall x (\vp(x;z) \leftrightarrow \neg \vp(x;y))$. If this does not already
hold for some given $\vp$, replace $\vp$ with $\theta(x;y,z) = \vp(x;y) \land \neg\vp(x;z)$. 
\end{itemize}
\end{defn}

\begin{conv} \label{localization-depends-on-T}
\emph{As the characteristic sequence is definable in $T$, its first-order properties
depend only on the theory and not on the model of $T$ chosen. Throughout this paper,
we will be interested in whether certain, possibly infinite, configurations appear 
as subgraphs of the $P_n$. By this we will always mean \emph{whether or not it is consistent with} $T$
that such a configuration exists when $P_n$ is interpreted in some sufficiently saturated model.
Thus, without loss of generality the formulas $P_n$ will often be identified 
with their interpretations in some monster model.} 
\end{conv}

Characteristic sequences give a natural context for studying the complexity of $\vp$-types, which
correspond in this case to complete graphs.

\begin{defn} \label{small-defn} Fix $T, \vp$, $M \models T$ and $(T, \vp) \mapsto \langle P_n \rangle$. 
\begin{enumerate}
\item A \emph{positive base set} is a set $A \subset P_1$ such that $A^n \subset P_n$ for all $n<\omega$.
\item The sequence $\langle P_n \rangle$ has \emph{support k} if: $P_n(y_1,\dots y_n)$
iff $P_k$ holds on every $k$-element subset of $\{ y_1,\dots y_n \}$. The sequence has
\emph{finite support} if it has support $k$ for some $k<\omega$.
\item The elements $a_1,\dots a_k \in P_1$ are a \emph{k-point extension} of the $P_\infty$-complete graph $A$ just in case $Aa_1,\dots a_k$ is also 
a $P_\infty$-complete graph.
\end{enumerate}
\end{defn}

\begin{obs} \label{pbs}
Fix $T, \vp$ and $M \models T$ and suppose $(T, \vp) \mapsto \langle P_n \rangle$.
\begin{enumerate} 
\item The following are equivalent, for a set $A \subset M$:
\begin{enumerate}
\item $A$ is a positive base set.
\item The set $\{ \vp(x;a) : a \in A \}$ is consistent.
\end{enumerate}
\item The following are equivalent, for a set $A \subset P_1$:
\begin{enumerate}
\item $A^n \cap P_n = \emptyset$ for some $n$.
\item $\{ \vp(x;a) : a \in A \}$ is $1$-consistent but $n$-inconsistent (Convention \ref{notation}(2)).
\end{enumerate}
Note that if $A$ is infinite, compactness then implies some instance of $\vp$ divides.
\item The following are equivalent:
\begin{enumerate}
\item $\langle P_n \rangle$ has finite support.
\item $\vp$ does not have the finite cover property. 
\end{enumerate}
\end{enumerate}
\end{obs}

Localization is a definable restriction of the predicates $P_n$ of a certain useful form which eliminates some of the combinatorial noise
around a positive base set $A$ under analysis. Definability ensures that
Convention \ref{localization-depends-on-T} applies when asking whether certain configurations are present in some localization. 

\begin{defn} \emph{(Localization, Definition 5.1 of \cite{mm-article-2})}
Fix a characteristic sequence $(T, \vp) \rightarrow \langle P_n \rangle$, and choose $B, A \subset M \models T$ with $A$ 
a positive base set, possibly empty.
A \emph{localization} $P^f_n$ of the predicate $P_n(y_1,\dots y_n)$ around the positive base set $A$ with parameters from $B$ is given by
a finite sequence of triples $f: m \rightarrow \omega \times \fss(y_1,\dots y_n) \times \fss(B)$
where $m<\omega$ and:
\begin{itemize}
\item writing $f(i) = (r_i, \sigma_i, \beta_i)$ and $\check{s}$ for the elements of the set $s$, we have:
\[  P^f_n(y_1,\dots y_n) :=  \bigwedge_{i\leq m} ~~ P_{r_i}(\check{\sigma_i}, \check{\beta_i})       \]
\item for each $\ell < \omega$, $T_1$ implies that there exists a $P_\ell$-complete graph $C_\ell$ such that 
$P^f_n$ holds on all $n$-tuples from $C_\ell$. 
If this last condition does not hold, $P^f_n$ is a \emph{trivial localization}. By \emph{localization} we will always mean
non-trivial localization.
\item In any model of $T_1$ containing $A$ and $B$, $P^f_n$ holds on all $n$-tuples from $A$. 
\end{itemize}
\end{defn}

For the purposes of this article, we will indicate where localization is useful without, generally, specifying the parameters
or the form involved, writing simply ``there exists a localization in which...'' or ``after localization...'' for short.
Because this may always be taken to include a fixed positive base set, the essential complexity of the type under analysis
is not lost.
Localization reveals a gap in the classification-theoretic complexity of $\vp$ and of $P_2$. 
\S \ref{section:order-genericity} below will shed light on this result: 

\begin{concl} \emph{(Conclusion 5.10 of \cite{mm-article-2})}  \label{ps-stable}
Suppose $T$ is simple, $(T, \vp) \mapsto \langle P_n \rangle$. Then for any $n<\omega$, 
and any partition of $y_1, \dots y_n$ into object and parameter variables, 
after localization the formulas $P_2(y_1, y_2), \dots P_n(y_1,\dots y_n)$ do not have the order property.
\end{concl}

It turns out that when $\vp$ is NIP one can always localize (without losing sight of the positive base set $A$ under analysis) 
so that any given finite initial segment of the characteristic sequence is a complete graph. 
In other words, the characteristic sequence is non-trivial in the presence of the independence property. 

\begin{theorem} \emph{(Theorem 6.17 of \cite{mm-article-2})} \label{char-stable} 
Let $\vp$ be a formula of $T$ and $\langle P_n \rangle$ its characteristic sequence. 
\begin{enumerate}
\item If $\vp$ is NIP, 
then for each positive base set $A \subset P_1$ and for each $n<\omega$,
there exists a localization $P^{f_n}_1 \supset A$ of $P_1$ which is a $P_n$-complete graph, 
i.e. $\{ y_1,\dots y_n \} \subset P^{f_n}_1 \rightarrow P_n(y_1,\dots y_n)$. 
\item If $\vp$ has IP, then for all $n<\omega$, $P_1$ contains a $P_n$-empty tuple.  
\end{enumerate}
\end{theorem}

Furthermore, when $\vp$ is simple unstable we may assume that after localization, in any given finite initial segment of
the characteristic sequence, there are uniform finite bounds on the size of empty graphs. 

\begin{theorem} \emph{(Theorem 6.24 of \cite{mm-article-2})}  \label{char-simple} 
Let $\vp$ be a formula of $T$ and $\langle P_n \rangle$ its characteristic sequence. 

\begin{enumerate}
\item If $\vp$ is simple, then for each $P_\infty$-graph $A \subset P_1$ and for each $n<\omega$,
there exists a localization $P^{f_n}_1 \supset A$ of $P_1$ in which there is a uniform finite bound on the size of
a $P_n$-empty graph, i.e. there exists $m_n$ such that $X \subset P_1$ and $X^n \cap P_n = \emptyset$ implies $|X| \leq m_n$.
\item If $\vp$ is not simple, then for all but finitely many $r<\omega$, $P_1$
contains an infinite $(r+1)$-empty graph.  
\end{enumerate}
\end{theorem}

The stage is now set as follows. The characteristic sequence of hypergraphs are a sequence of incidence relations
defined on the parameter space of a formula $\vp$. Positive base sets correspond naturally
(though not necessarily uniquely) to base sets for $\vp$-types. 
We turn to the study of the generic interrelationships between sets generally, and positive base sets particularly, in the parameter space of 
a given $\vp$. Theorem \ref{char-stable} strongly focuses our attention on the ``wild'' case of theories
with the independence property and Theorem \ref{char-simple} suggests simple unstable theories as a first object of study.

\section{Counting functions on simple $\vp$} \label{section:counting}

\noindent Throughout this section, we consider the binary edge relation $P_2$ from the characteristic sequence of $\vp$. The notation and
vocabulary follow Convention \ref{notation}. 

\begin{obs}
Suppose $\vp$ is stable. Then after localization, for any two disjoint finite $X, Y \subset P_1$,
$\delta(X,Y)$ = 1. On the other hand, if $\vp$ is simple unstable then $P_1$ contains an empty pair. 
\end{obs}

\begin{proof}
Theorem \ref{char-stable}(1) says that when $\vp$ is stable, after localization $P_1$ is a complete graph, so a fortiori
there are no edges omitted between disjoint components. The second clause is Theorem \ref{char-stable}(2).
\end{proof}

\begin{defn}
Define $\alpha: \omega \rightarrow \omega$ to be
\[    \operatorname{max}~~ \{ \hat{e}(X) : X \subset P_1, |X| = n\} \]
i.e. the largest number of $P_2$-edges omitted over an $n$-size subset of $P_1$.
\end{defn}

\begin{obs} \label{simple-not-maxl}
Suppose $\vp$ is simple, i.e., $\vp$ does not have the tree property. 
Then after localization $\alpha(n) < \frac{n(n-1)}{2}$.
\end{obs} 

\begin{proof}
The maximum possible value $\frac{n(n-1)}{2}$ of any $\alpha(n)$ is attained on a $P_2$-empty graph, on which
$x \neq y \implies \neg P_2(x,y)$.
Apply Theorem \ref{char-simple} which says that when $\vp$ does not have the tree property 
then we have, after localization, a uniform finite bound $k$ on the size of a $P_2$-empty graph $X \subset P_1$.
So the function $\alpha$ is eventually strictly below the maximum.  
\end{proof}

\begin{cor} The function $\alpha(n)$ is meaningful, i.e. after localization
\[ \frac{n(n-1)}{2} > ~~\alpha(n)~~ > 0\] 
precisely when $\vp$ is simple unstable. 
\end{cor}

\noindent With some care we can easily restrict the range further. A famous theorem of Tur\'an says:

\begin{thm-lit} \emph{(Tur\'an, \cite{rsurvey}:Theorem 2.2)} \label{turan-1}
If $G_n$ is a graph with $n$ vertices and

\[  e(G) > \left( 1 - \frac{1}{k-1} \right) \frac{n^2}{2}                               \]

then $G_n$ contains a complete subgraph on $k$ vertices.  
\end{thm-lit}

\begin{defn}
$X = \langle a^t_i : t<2, i<\omega \rangle \subset P_1$ is an \emph{$(\omega, 2)$-array} if
for all $n<\omega$, 
\[ P_n(a^{t_1}_{i_1},\dots a^{t_n}_{i_n}) ~\iff~ (\forall j, \ell \leq n) \left( i_j = i_\ell \implies t_j = t_\ell \right) \] 
\end{defn}

\begin{claim} \emph{(Claim 4.5 of \cite{mm-article-2})} 
The following are equivalent, for a formula $\vp$ with characteristic sequence $\langle P_n \rangle$:
\begin{enumerate}
\item $\vp$ has the independence property.
\item $\langle P_n \rangle$ has an $(\omega, 2)$-array.
\end{enumerate}
\end{claim}

\begin{obs} \label{right-count}
Suppose that $\langle P_n \rangle$ has an $(\omega, 2)$-array. 
Then $\alpha(n) \geq \left\lfloor \frac{n}{2} \right\rfloor$.
\end{obs}

\begin{cor} \label{lower-bound}
When $\vp$ is simple unstable, then after localization
\[ \left(1-\frac{1}{k-1}\right)\frac{n^2}{2} \geq \alpha(n) \geq \left\lfloor \frac{n}{2} \right\rfloor\] 
\end{cor}

\begin{proof}
If $\vp$ is simple unstable, $\vp$ has the independence property and so $P_1$ contains an $(\omega, 2)$-array; so 
the righthand side is Observation \ref{right-count}. For the lefthand side, let $k > 1$ be the uniform finite bound on the size of an empty graph
from Theorem \ref{char-simple}, and apply Tur\'an's theorem to the dual graph. 
\end{proof}

At the end of Section \ref{section:regularity} we will give a proof of the following: 

\begin{propn} \label{counting-thm}
When $\vp$ is simple unstable either
\[ \left(1-\frac{1}{1-k}\right)\frac{n^2}{2} ~\geq~ \alpha(n) ~\geq~ \frac{n^2}{4} ~\mbox{\hspace{8mm}\emph{or}\hspace{8mm}}~ \mathcal{O}(n^2) > \alpha(n) \geq \left\lfloor \frac{n}{2} \right\rfloor \] 
\end{propn}

The proof will follow from Theorem \ref{counting-alpha-2} below, which will show more, namely that for $\vp$ simple unstable, either
$\mathcal{O}(n^2) > \alpha(n)$ or there exists an infinite empty pair in $P_1$. 

Our strategy is going to be to show that in the absence of such an ``empty pair'' 
we can repeatedly partition sufficiently large graphs into many pieces of roughly equal size 
in such a way that, at each stage, the bulk of the omitted edges must occur inside the (eventually, much smaller) pieces. 
The main tool will be Theorem \ref{szr} below. 

\section{Szemer\'edi regularity} \label{section:regularity}

We begin with a review of Szemer\'edi's celebrated \emph{regularity lemma}. 
Recall that $\epsilon, \delta$ are real numbers, $0 < \epsilon < 1$ and $0 \leq \delta \leq 1$, following Convention \ref{notation}.

\begin{defn} \label{graph-regularity} \cite{sz}, \cite{rsurvey} 
The finite bipartite graph $(X,Y)$ is $\epsilon$-regular if for every $X^\prime \subset X$, $Y^\prime \subset Y$
with $|X^\prime| \geq \epsilon|X|$, $|Y^\prime| \geq \epsilon|Y|$, we have: $|\delta(X,Y) - \delta(X^\prime, Y^\prime)|<\epsilon$.
\end{defn} 

The regularity lemma says that sufficiently large graphs can always be partitioned 
into a fixed finite number of pieces $X_i$ of approximately equal size 
so that almost all of the pairs $(X_i, X_j)$ are $\epsilon$-regular. 

\begin{thm-lit} \label{szr} \emph{(Szemer\'edi's Regularity Lemma \cite{rsurvey}, \cite{sz})}
For every $\epsilon, m_0$ there exist $N = N(\epsilon, m_0)$, $m = m(\epsilon, m_0)$ such that 
for any graph $X$, $N \leq |X| < \aleph_0$, for some $m_0 \leq k \leq m$ there exists a partition $X = X_1 \cup \dots \cup X_k$
satisfying:

\begin{itemize}
\item $| |X_i| - |X_j| | \leq 1$ for $i,j \leq k$
\item All but at most $\epsilon k^2$ of the pairs $(X_i, X_j)$ are $\epsilon$-regular.
\end{itemize}
\end{thm-lit}

One important consequence is that we may, approximately, describe large graphs $G$ as random graphs
where the edge probability between $x_i$ and $x_j$ is the density $d_{i,j}$ between components $X_i, X_j$ in some
Szemer\'edi-regular decomposition. We include here two formulations of this idea from the literature, the first for intuition
and the second for our applications.

\begin{thm-lit} \emph{(from Gowers \cite{gowers})} \label{gowers-1}
For every $\alpha > 0$ and every $k$ there exists $\epsilon > 0$ with the following property. Let
$V_1,\dots V_k$ be sets of vertices in a graph $G$, and suppose that for each pair $(i,j)$ the
pair $(V_i, V_j)$ is $\epsilon$-regular with density $\delta_{ij}$. Let $H$ be a graph with vertex set
$(x_1,\dots x_k)$ and let $v_i \in V_i$ be chosen uniformly at random, the choices being independent. Then 
the probability that $v_iv_j$ is an edge of $G$ iff $x_ix_j$ is an edge of $H$ differs from 
$\Pi_{x_ix_j\in H} \delta_{ij} \Pi_{x_ix_j\notin H} (1-\delta_{ij})$ by at most $\alpha$. 
\end{thm-lit}

The formulation we will use, Theorem \ref{key-lemma}, requires a preliminary definition.

\begin{defn} \cite{rsurvey} \emph{(The reduced graph)}
\begin{enumerate}
\item Let $G = X_1, \dots X_k$ be a partition of the vertex set of $G$ into disjoint components.
Given parameters $\epsilon, \delta$, define the \emph{reduced graph} $R(G,\epsilon, \delta)$ to be 
the graph with vertices $x_i$ $(1\leq i \leq k)$ and an edge between $x_i, x_j$ just in case the pair
$(X_i, X_j)$ is $\epsilon$-regular of density $\geq \delta$.

\item Write $R(t)$ for a full graph of height $t$ whose reduced graph is $R$, i.e., 
$R(t)$ consists of $k$ components $X_1,\dots X_k$, each with $t$ vertices, such that
$e(X_i)=0$, and $\delta(X_i, X_j)=1$ iff there is an edge between $x_i$ and $x_j$ in $R$. 
\end{enumerate}
\end{defn}

The following lemma (called the ``Key Lemma'' in \cite{rsurvey}) says that sufficiently small subgraphs
of the reduced graph must actually occur in the original graph $G$. 

\begin{thm-lit} \emph{(Key Lemma, \cite{rsurvey}:Theorem 2.1)} \label{key-lemma}
Given $\delta > \epsilon > 0$, a graph $R$, and a positive integer $m$, let $G$ be any graph whose
reduced graph is $R$, and let $H$ be a subgraph of $R(t)$ with $h$ vertices and maximum degree $\Delta > 0$.
Set $d = \delta - \epsilon$ and $\epsilon_0 = d^{\Delta}/(2+\Delta)$. Then
if $\epsilon \leq \epsilon_0$ and $t-1 \leq \epsilon_0 m$, then $H \subset G$. Moreover
the number of copies of $H$ in $G$ is at least $(\epsilon_0 m)^h$.  
\end{thm-lit}

\begin{rmk} \label{reft}
In the statement of the Key Lemma, ``$H \subset G$'' means that
there is a bijection $f: H \rightarrow X \subset G$ such that $e(h_1,h_2)$ implies $e(f(h_1), f(h_2))$. 
With some slight modifications (recording whether a missing edge in the reduced graph means the density is near 0
or the pair is not regular; and using the dual graphs when necessary) 
we may assume ``$H \subset G$'' has the usual meaning of isomorphic embedding, but this will not be 
an issue for the arguments in this section.
\end{rmk}

We now work towards a proof of Proposition \ref{counting-thm}.

\begin{conv} \label{interstitial} \emph{(Interstitial edges, $b_{\epsilon, \ell}$, $N_{\epsilon, \ell}$, $E_{\epsilon, \ell}$)}
\begin{enumerate}
\item Let $G$ be a graph and let $G = X_1 \cup \dots \cup X_n$ be a decomposition into disjoint components, for instance as given 
by Theorem \ref{szr}. Call any edge between vertices $x \in X_i, z \in X_j$, $i\neq j$ an \emph{interstitial edge}. 
\item Let $b_{\epsilon, \ell}$ denote the upper bound on the necessary number of components, given by the regularity lemma
as a function of $\epsilon, \ell$.
\item Write \emph{$(\epsilon, \ell)^*$-decomposition} to denote any Szemer\'edi-regular decomposition \emph{into $k$ components},
for any $\ell \leq k \leq b_{\epsilon, \ell}$. 
\item Let $N_{\epsilon, \ell}$ denote the threshold size given by the regularity lemma as a function of
$\epsilon, \ell$, such that any graph $X$ with $|X| > N_{\epsilon, \ell}$ admits an $(\epsilon, \ell)^*$-decomposition. 
\item Let $E_{\epsilon, \ell} \subset \omega$ ($E$ for ``exactly'')
be the (possibly empty) set of cardinalities $n$ for which $|X| = n$ implies $X$ admits a Szemer\'edi-regular decomposition
into \emph{exactly} $\ell$ components. See the next Remark. 
\end{enumerate}
\end{conv}

\begin{rmk} \emph{On Definition \ref{interstitial}(2)-(4): the Regularity Lemma, along with the pigeonhole principle,
implies that for cofinally many $\ell$, $E_{\epsilon, \ell}$ is infinite. Often, as Corollary \ref{count}(3) suggests,
for the purposes of our asymptotic argument it is sufficient to know that the number of components fluctuates in a 
certain fixed range, as given by the Regularity Lemma.}
\end{rmk}

An easy application of the Key Lemma shows that: 

\begin{obs} \label{k-l-apply}
Suppose that there exists $\delta$, $0<\delta<1$ such that for all $0 <\epsilon <1$ and all 
$N \in \mathbb{N}$ there exist disjoint subsets $X_N, Y_N \subset P_1$, $|X_N|=|Y_N| \geq N$
such that $(X_N,Y_N)$ is $\epsilon$-regular with density $\delta$. Then $P_1$ contains an infinite empty pair.
\end{obs}

\begin{proof}
Apply the Key Lemma to each dual graph $(X_N,Y_N)^\prime$, which is still regular and whose density remains bounded away from $0$ and $1$. 
For each $t <\omega$, 
for all $N$ sufficiently large, $(X_N, Y_N)^\prime$ contains a complete bipartite graph on $t$ vertices, as this occurs as a subgraph of $R(t)$. 
\end{proof}

\begin{lemma} \label{count}
Suppose that $P_1$ does not contain an infinite empty pair. 

\begin{enumerate}
\item There is a function $f:(0,1)\times \omega \rightarrow (0,1)$ monotonic which approaches $1$
as $\epsilon \rightarrow 0$ and $N \rightarrow \infty$ such that if $(X,Y)$ is an $\epsilon$-regular pair with $|X|=|Y|=n$ then 
$d(X,Y) \geq f(\epsilon, N)$. 

\item There is a function $g: ((0,1)\times \omega) \times \omega \rightarrow (0,1)$, which is defined on all $((\epsilon, \ell), n)$ for which
$n \geq N_{\epsilon, \ell}$, and which is monotonically increasing and approaches $1$ as $(\epsilon, \ell)$ stays fixed and $n \rightarrow \infty$, 
such that if $|X| = n$ then the density between any two regular components in an $(\epsilon, \ell)^*$-decomposition of $X$ 
is at least $g((\epsilon, \ell), n)$.

\item For every constant $c > 0$, and for all $\epsilon_0 > 0$, there exist $0 < \epsilon < \epsilon_0$ 
and for each such $\epsilon$, cofinally many $\ell < \omega$ such that:  
for all $n$ sufficiently large and all graphs $X$ with $|X| = n$, 
the number of interstitial edges in any $(\epsilon, \ell)^*$-decomposition of $X$ is strictly less than $cn^2$.
\end{enumerate}
\end{lemma}

\begin{proof}
(1) This restates Observation \ref{k-l-apply}. 

(2) The regularity lemma provides a decomposition in which all components are approximately the same size $(\pm 1)$, so the density 
of each $\epsilon$-regular pair will be at least $f(\epsilon, \frac{n}{b_{\epsilon, \ell}})$. 

It remains to prove (3).
For the moment, let $\epsilon, \ell$ be arbitrary and suppose that $|X| > N_{\epsilon, \ell}$. Then 
$|X| = n$ admits an $\epsilon$-regular decomposition into at $k$-many pieces, each of size approximately $m = \frac{n}{k}$,  
where $(\dagger) ~~ \ell \leq k \leq \ell^\prime := b_{\epsilon, \ell}$.

Writing $\delta := g((\epsilon, \ell), \frac{n}{\ell^\prime})$, the contribution of the interstitial edges is at most:
\[  \epsilon k^2 m^2 +  (1-\epsilon)(k)^2 \left(1-\delta\right)m^2  \]
\noindent where the term on the left assumes the irregular pairs are empty, and the term on the right
counts the expected number of interstitial edges missing from the regular pairs. By $(\dagger)$, 
this in turn is bounded by:
\begin{align*} \leq & ~~ \epsilon (\ell^\prime)^2 m^2 +  (1-\epsilon)(\ell^\prime)^2 \left(1-\delta\right)m^2 \\ 
  \leq & ~~ \epsilon (\ell^\prime)^2 \left(\frac{n}{l}\right)^2 +  (1-\epsilon)(\ell^\prime)^2 \left(1-\delta\right) \left(\frac{n}{l}\right)^2 \\
  \leq & ~~ n^2 \left( \frac{\ell^\prime}{\ell} \right)^2 \left(\vrt \epsilon + (1-\epsilon)\left(1-\delta\right) \right) 
\end{align*}

\noindent Thus our claim will hold whenever $\epsilon + (1-\epsilon)(1-\delta) < c \frac{\ell}{\ell^\prime}$. To obtain this,
choose $\epsilon > 0$ as small as desired and $\ell$ as large as desired. Then $\delta$ is monotonically increasing
and approaches $1$ as a function of $n$ by (2), as desired.
\end{proof}

We are now prepared to prove:

\begin{theorem} \label{counting-alpha-2}
When $\vp$ is simple unstable, if there does not exist an infinite empty pair $X, Y \subset P_1$,
then $\alpha(n) < \mathcal{O}(n^2)$.
\end{theorem}

\begin{proof}
Given a positive real constant $c_0 > 0$, choose $c, k, t$ such that $0 < c < 1$, $k, t \in \mathbb{N}$ and 
$c_0 > 2c + \frac{1}{k^t}$. 
Fix a pair $(\epsilon, \ell)$ such that $\ell > k$ and $(\epsilon, \ell)$ is one of the cofinally many pairs
described in Lemma \ref{count}(3) for the constant $c$. 
Now, for $n$ sufficiently large and any $|X| = n$, 
each of the components in an $(\epsilon, \ell)^*$-decomposition of $X$ will have size $> N_{\epsilon, \ell}$ 
so will themselves admit an $(\epsilon, \ell)^*$-decomposition with few interstitial edges. Repeating this argument to an arbitrary depth 
we can confirm that the bulk of the ostensibly missing edges must continually vanish inside the 
(eventually) relatively much smaller components at each successive decomposition.

More precisely, let $\ell^{\prime} := b_{\epsilon, \ell}$ and suppose $n >> (N_{\epsilon, \ell})(\ell^{\prime})^t$.  
By repeated application of 
Lemma \ref{count}(3) to each successive decomposition, we obtain the following upper bound on $\alpha(n)$,
where $\ell \leq k_i \leq \ell^\prime$ for each $1 \leq i \leq t$.
The rightmost term assumes that after $t-1$ levels of decomposition we obtain components which are themselves empty graphs. 
\begin{align*}
\alpha(n) ~~ < & ~~ cn^2 + c\left(\frac{n}{k_1}\right)^2 k_1 + c\left(\frac{n}{(k_2)^2}\right)^2 (k_2)^2 + \\
 & \hspace{30mm} \dots + c\left(\frac{n}{(k_{t-1})^{t-1}}\right)^2 (k_{t-1})^{t-1} + \left(\frac{n}{(k_t)^t}\right)^2 (k_t)^t \\
 < & ~~ cn^2\left( 1 + \frac{1}{k_1} + \frac{1}{(k_2)^2} + \dots + \frac{1}{(k_{t-1})^{t-1}} \right) + \left(\frac{n^2}{(k_t)^t}\right) \\
 < & ~~ cn^2\left( 1 + \frac{1}{\ell} + \frac{1}{\ell^2} + \dots + \frac{1}{\ell^{t-1}} \right) + \left(\frac{n^2}{\ell^t}\right) \\
 < & ~~ n^2 \left( \frac{\ell c}{\ell-1} + \frac{1}{\ell^t} \right) \\
 < & ~~ \left(2c + \frac{1}{\ell^t}\right) n^2 ~~ < ~~ c_0 n^2 \\
\end{align*}
\noindent by summing the convergent series. We have shown that for any constant $c_0$, for all $n$ sufficiently large
$\alpha(n) < c_0 n^2$, so we finish.
\end{proof}

\begin{proof} (of Proposition \ref{counting-thm})
This is now an immediate corollary of Theorem \ref{counting-alpha-2}, 
$\frac{n^2}{4}$ being the number of edges omitted in an empty pair.
\end{proof}

\begin{rmk}
Theorem \ref{counting-alpha-2}, and thus Proposition \ref{counting-thm}, are more natural than might appear. 
On one hand, as Szemer\'edi regularity deals with density, it cannot (in this formulation) give precise information about edge counts below $\mathcal{O}(n^2)$. On the other, the random graph contains many infinite empty pairs, 
for instance $(\{ (a,z) : z \in M, z \neq a \}, \{ (y,a) : y \in M, y \neq a \})$ when $\vp = xRy \land \neg xRz$. 
One could imagine a future use for such theorems in suggesting ways of decomposing the parameter spaces of simple formulas into
parts whose structure resembles random graphs (with many overlapping empty pairs) and
parts whose structure is more cohesive, indicated by $\alpha(n) < \mathcal{O}(n^2)$.
\end{rmk}

\section{Order and genericity} \label{section:order-genericity}

In this section we show that $P_1$, after localization, 
admits arbitrarily large $\epsilon$-regular pairs of some fixed density bounded away from $0$ and $1$
precisely when $P_2$, after localization, is unstable. Compare Conclusion \ref{ps-stable}.

\begin{obs} \label{slice-order}
Suppose that for some $0< \delta < 1$ and for all $\epsilon, n$ with $0<\epsilon<1$, 
$n \in \mathbb{N}$ we have a bipartite $R$-graph $(X,Y)$, $|X| = |Y| \geq n$, such that
$(X,Y)$ is $\epsilon$-regular with density $d$, where $|d - \delta| < \epsilon$. Then 
$R$ has the order property. 
\end{obs}

\begin{proof}
It suffices to show that for arbitrarily small $\epsilon_0$ 
and arbitrarily large $k_0$ there is a Szemer\'edi-regular decomposition of $X$ and of $Y$ into $k_0$ pieces
such that all but $k_0 (\epsilon_0)^2$ of the pairs $X_i$, $Y_i$ are $\epsilon_0$-regular with density near $\delta$.
This is because we can apply the Key Lemma (in light of Remark \ref{reft}) to conclude that any pattern which 
appears in the reduced graph corresponding to these components, in particular some given fragment of the order property, 
actually occurs in $X,Y$. The subtlety is to ensure that the densities of the regular pairs 
are all approximately the same. 

Given $\epsilon_0$, $k$, let $k_0$, $N_0$ be the number of components and threshold size, respectively, given by the regularity lemma. Choose $\epsilon$ so that $\frac{1}{k_0} > \epsilon$ and $n > N_0$. Let $(X,Y)$ be the $\epsilon$-regular pair of size at least $n$ and density near $\delta$, given by hypothesis. 

By regularity, $n > N_0$ means that there is a decomposition $X = \cup_{i\leq k_0} X_i$,
$Y = \cup_{i\leq k_0} Y_i$ into disjoint pieces of near equal size and that all but 
$\epsilon_0(k_0)^2$ of the pairs $(X_i, Y_j)$ are $\epsilon_0$-regular. 
However any one of these regular pairs $(X_i, Y_j)$ will satisfy $|X_i|, |Y_j| = n/k_0 > \epsilon n$, 
so $|d(X_i, Y_j) - d(X,Y)| = |d(X_i, Y_j) - \delta \pm \epsilon| < \epsilon$
and $|d(X_i, Y_j) - \delta | < 2\epsilon$, as desired. 
\end{proof}

Recall that an equivalent definition of $SOP$ is that there exists an indiscernible sequence
$\langle a_i : i<\omega \rangle$ on which $\exists x (\neg \vp(x;a_j) \land \vp(x;a_i)) \iff j<i$. 
The main step in Shelah's classic proof that any unstable theory which does not have the independence property
must have the strict order property can be characterized as follows:

\begin{thm-lit} \emph{(Shelah)} \label{threaten-order}
Let $c$ be a finite set of parameters and $\langle a_i : i<\omega \rangle$ a $c$-indiscernible sequence. 
For $n<\omega$, any formula $\theta(x;\overline{z})$ and relations $R(x;y)$, $R_1,\dots R_n$ 
where $\ell(y) = \ell(a_i)$ and $R_i \in \{ R(x;y), \neg R(x;y) \}$ for $i\leq n$, if
\[ i_1 < \dots < i_n \implies \exists x \left( \vrt \theta(x;c) \land R_1(x;a_{i_1}) \land \dots \land R_n(x;a_{i_n}) \right) \]
then either
\begin{itemize}
\item $\exists x \left( \vrt\theta(x;c) \land R_1(x;a_{i_{\sigma(1)}}) \land \dots \land R_n(x;a_{i_{\sigma(n)}}) \right)$ for any 
permutation $\sigma : n \rightarrow n$
\item some formula of $T$ has the strict order property. 
\end{itemize}
\end{thm-lit}

The idea is to express the permutation $\sigma$ as a sequence of swaps of successive elements (in the sense of the order $<$),
and use the first instance, if any, where the swap produces inconsistency to obtain a sequence witnessing strict order. 
For details, see \cite{Sh:c}, Theorem II.4.7, pps. 70--72.

The subtlety in the corollary is to obtain not just the independence property but a bipartite random graph. 

\begin{cor} \label{op-to-ip-2}
Suppose that $R(x;y)$ has the order property. If $T$ does not have the strict order property, 
then there exist infinite disjoint sets $A, B$ on which $R$ is a bipartite random graph (i.e., 
$R(x;y)$ is $I^2_2$ in the sense of Definition \ref{constellations} below).
\end{cor}

\begin{proof}
We first fix a template. Let $M$ be a countable model of the theory of a bipartite random graph with two
sorts $P,Q$ and a single binary edge relation $E(x;y)$ with $E(x;y) \implies P(x) \land Q(y)$. 
Let $\langle x_i : i<\omega \rangle$, $\langle y_i : i<\omega \rangle$ be an enumeration of $P$ and $Q$,
respectively. 

Now let $\langle a_i b_i : i<\omega \rangle$ be an indiscernible sequence on which $R$ has the order property, i.e.
$R(a_i, b_j) \iff i<j$. Suppose that for every $i<\omega$ we could find an element $c_i$ such that
for all $j<\omega$, $R(c_i, b_j) \iff E(x_i, y_j)$ in the template. Then setting $C := \bigcup_{i<\omega} c_i$,
$B := \bigcup_{j<\omega} b_j$, $(C,B)$ is a bipartite random graph for $R$. 

So it remains to show that any finite subset $p$ of the type $p_i(x) \in S(B)$ of any such $c_i$ is consistent. 
Let $\eta, \nu$ be disjoint finite subsets of $\omega$, 
and let $p(x) = \bigwedge_{j \in \eta} R(x;b_j) \land \bigwedge_{k \in \nu} \neg R(x;b_k)$.  
We are now in a position to apply Theorem \ref{threaten-order}; as $T$ is $NSOP$, $p(x)$ must be consistent. 
\end{proof}

\begin{defn} Fix a binary edge relation $R$. 
Call a density $0 \leq \delta \leq 1$ \emph{attainable} if for all $\epsilon$
there exists a sequence $\langle S^\delta_{\epsilon} = \langle (X_i, Y_i) : i<\omega \rangle$
of finite bipartite $R$-graphs such that
for all $n<\omega, \epsilon > 0$ there is $N < \omega$ such that for all $i > N$, 

\begin{itemize}
\item $|X_i| = |Y_i| \geq n$,
\item $(X_i, Y_i)$ is $\epsilon$-regular with density $d_i$, where
$|d_i - \delta| <\epsilon$.
\end{itemize}
\end{defn}

\begin{concl} \label{excluded-middle}
Assume $T$ does not have the strict order property. Then
the following are equivalent for a binary relation $R(x,y)$:

\begin{enumerate}
\item For some $0 < \delta < 1$ and for all $N, \epsilon$ there exist
disjoint $X, Y$ with $|X| = |Y| \geq N$ such that $(X,Y)$ is $\epsilon$-regular
with density $d$, $|d - \delta| < \epsilon$.

\item For any attainable $0 < \delta < 1$ such that for all $N, \epsilon$ there exist
disjoint $X, Y$ with $|X| = |Y| \geq N$ such that $(X,Y)$ is $\epsilon$-regular
with density $d$, $|d - \delta| < \epsilon$.

\item $R$ has the order property. 

\end{enumerate}
\end{concl}

\begin{proof}
(2) $\rightarrow$ (1) Attainable densities exist, e.g. $\frac{1}{2}$: consider subgraphs of an infinite random bipartite 
graph. 

(1) $\rightarrow$ (3) Observation \ref{slice-order}. 

(3) $\rightarrow$ (1) Corollary \ref{op-to-ip-2}, which says that from (3), assuming $NSOP$,
we can construct an infinite random bipartite graph with edge relation $R$.
\end{proof}

In other words, regularity plus compactness implies that density bounded away from $0,1$ gives any bipartite
configuration including the order property, and model theory implies that the order property is enough to reverse the argument. 

\begin{cor} \label{c-excluded-middle}
Assume $T$ does not have the strict order property, and $(T,\vp) \mapsto \langle P_n \rangle$. 
Then the following are equivalent:

\begin{enumerate}
\item After localization, $P_2$ does not have the order property.

\item After localization, the density of any sufficiently large 
$P_2$-regular pair $(X,Y)$ must approach either 0 or 1. More precisely,
there exists $f: \mathbb{N} \times (0,1) \rightarrow [0,\frac{1}{2}]$ 
monotonic increasing as $n \rightarrow \infty$, $\epsilon \rightarrow 0$ such that 
if $X, Y \subset P_1$, $|X|, |Y| \geq n$ and $(X,Y)$ is $\epsilon$-regular, then either
$d(X,Y) < f(n,\epsilon)$ or $d(X,Y) > 1-f(n,\epsilon)$.  
\end{enumerate}
\end{cor}

\begin{cor} If $T$ is simple, then any characteristic sequence associated to one of its formulas
satisfies the equivalent conditions of Corollary \ref{c-excluded-middle}. 
\end{cor}

\begin{proof}
Conclusion \ref{ps-stable}.
\end{proof}

\begin{rmk} The class of theories satisfying the equivalent conditions of 
Corollary \ref{c-excluded-middle} strictly contains the simple theories. 
Example 3.6 of \cite{mm-article-2} gives a formula with the tree property whose $P_2$ does not have the order property.
This is essentially $T^*_{feq}$ from \cite{ShUs}; basic examples of $TP_2$ will work. 
\end{rmk}

\begin{rmk}
Any formula with $SOP_2$, also called $TP_1$, has the order property in $P_2$. For $SOP_2$, see \cite{ShUs}.
However, the next section suggests that more precise order properties may be useful.  
\end{rmk}

\section{Two kinds of order property} \label{section:two-order-properties}

Towards understanding the role of instability in the characteristic sequence,
this section considers two polar opposite order properties and their implications for $P_2$. 

\begin{defn} \label{two-op-defn}
\emph{(Two kinds of order property)} Let $\langle P_n \rangle$ be the characteristic sequence of $\vp$.
\begin{enumerate}
\item $\vp$ has the \emph{$n$-compatible order property}, for some $n<\omega$ (or $n=\infty$) if
there exist $\langle a_i, b_i : i<\omega \rangle$
such that for all $m\leq n$ (or $m<\omega$), $P_{2m}(a_{i_1}, b_{j_1},\dots a_{i_m}, b_{j_m})$ iff 
$\operatorname{max} \{ i_1,\dots i_m \} < \operatorname{min} \{ j_1, \dots j_m \}$.
\item[(1)$^\prime$] When the sequence has support $2$ this becomes:
there exist $\langle a_i, b_i : i<\omega \rangle$ such that $P_2(a_i,a_j)$, $P_2(b_i, b_j)$
for all $i,j$ and $P_2(a_i, b_j)$ iff $i<j$. 

\item $\vp$ has the \emph{$n$-empty order property}, for some $n \in \omega$, if:
\\ there exist $\langle a_i, b_i : i<\omega \rangle$
such that (i) $P_2(a_i;b_j)$ iff $i < j$ and (ii) $\neg P_n(a_{i_1},\dots a_{i_n})$, $\neg P_n(b_{i_1},\dots b_{i_n})$ hold 
for all $i_1,\dots i_n <\omega$.  
\end{enumerate}
\end{defn}

Let us briefly justify not focusing on a natural third possibility, the ``semi-compatible order property,''
in which the elements $\langle a_i : i<\omega \rangle$ are an empty graph and the elements $\langle b_i : i<\omega \rangle$
are a positive base set.  

\begin{obs}
There is a formula in the random graph which has the semi-compatible order property. 
\end{obs}

\begin{proof}
Choose two distinguished elements $0$, $1$ (this can be coded without parameters).
Define $\psi(x;y,z)$ to be $x=y$ if $z=0$, $xRy$ otherwise. Then on any sequence of distinct elements
$\langle a_i b_i : i<\omega \rangle \subset M$ which witness the order property ($a_i R b_j \iff i<j$), 
we have additionally that
\[ \exists x \left(\svrt \psi(x;a_i,0) \land \psi(x;b_j,1)\right) \iff \exists x \left(\svrt x=a_i \land x R b_j\right) \iff i<j \]
so $P_2$ has the order property on the sequence $\langle (a_i,0), (b_i,1) : i<\omega \rangle$.
On the other hand, $\exists x (x=a_i \land x=a_j) \iff i=j$, so the row of elements $(a_i,0)$ is a $P_2$-empty graph.
Finally, $\exists x (xRb_i \land xRb_j)$ always holds, by the axioms of the random graph; so the row of elements
$(b_j, 1)$ is a $P_\infty$-complete graph.  
\end{proof}

\begin{obs}
There is a formula in a simple rank 3 theory which has the $2$-empty order property.
\end{obs}

\begin{proof}
Let $T$ be the theory of two crosscutting equivalence relations, $E$ and $F$, each with infinitely many infinite classes
and such that each intersection $\{ x : E(a,x) \land F(x,b) \}$ is infinite. Let $P$ be a unary predicate such that
\begin{itemize}
\item $(\forall x,y) ( E(x,y) \land F(x,y) \implies P(x) \iff P(y))$
\item For all $n<\omega$ and $y_1,\dots y_k, y_{k+1},\dots y_n$ elements of distinct $E$-equivalence classes, 
there exists $z$ such that $i\leq k \implies (\forall x)(E(x,y) \land F(x,z) \implies P(x))$ and 
$k<i\leq n \implies (\forall x)(E(x,y) \land F(x,z) \implies \neg P(x)))$ 
\end{itemize}

Let $\psi(x;y,z)$ be $E(x,y) \land P(y)$ if $z=0$, and $F(x,y) \land P(y)$ otherwise. Let $\langle a_i, b_i : i<\omega \rangle$ 
be a sequence of elements chosen so that $(\forall x) (E(x,a_i) \land F(x,b_j) \implies P(x))$ iff $i<j$. Then it is easy to see 
$\psi$ has the $2$-empty order property on the sequence $\langle (a_i,0), (b_i,1) : i<\omega \rangle$.
\end{proof}

\begin{rmk} \label{half-comp}
Assuming $MA + 2^{\aleph_0} > \aleph_1$, Shelah has constructed an ultrafilter on $\omega$ which saturates (small) models of the
random graph, but not of theories with the tree property \emph{(\cite{Sh:c} Theorem VI.3.9)}. 
This is a strong argument for the ``semi-compatible order property'' being less complex: it cannot, by itself,
imply maximality in the Keisler order, whereas we will see that the $\infty$-compatible order property does. 
It may still be that persistence, in the sense of \cite{mm-article-2}, of any order property in $P_2$ creates complexity.
\end{rmk}

We return to the study of the compatible order property.

\begin{conv}
When more than one characteristic sequence is being discussed, write $P_n(\vp)$ to indicate the $n$th hypergraph associated 
to the formula $\vp$. Recall that $\vp_\ell$ is shorthand for $\bigwedge_{1\leq i\leq \ell} \vp(x;y_i)$. 
\end{conv}

The following general principle will be useful.

\begin{claim} \label{comp-claim}
Suppose that we have a sequence $C :=\langle c_i : i \in \mathbb{Z} \rangle$ and a formula $\rho(x;y,z)$ such that:
\begin{enumerate}
\item $\exists x \rho(x;c_i,c_j) \iff i<j$
\item $\exists x \left(\bigwedge_{\ell \leq n} \rho(x;c_{i_\ell},c_{j_\ell})\right)$ just in case 
$\operatorname{max} \{ i_1,\dots i_n \} < \operatorname{min} \{ j_1, \dots j_n \}$
\end{enumerate}

Then $\rho$ has the $\infty$-compatible order property.
\end{claim}

\begin{proof}
By compactness, it is enough to show that there are elements $\langle \alpha_i, \beta_i : i<n \rangle$ witnessing
a fragment of the $\infty$-compatible order property of size $n$. 

Define $\alpha_1\dots \alpha_n, \beta_1,\dots \beta_n$ as follows. Remark \ref{cop-picture} provides a picture. 

\begin{itemize}
\item $\alpha_i := c_{2i-1} c_{4n-2i+1}$, $1 \leq i \leq n$
\item $\beta_i := c_{-2i} c_{2_i}$, $1 \leq i \leq n$
\end{itemize}

Then $P_1(\alpha_i)$, $P_1(\beta_i)$ for $1\leq i\leq n$ by (1). For all $1 \leq k, r \leq n$ with $r+k=m$,
condition (2) says that $P_m(\alpha_{i_1},\dots \alpha_{i_k}, \beta_{j_1},\dots \beta_{j_r})$
iff
\[ \operatorname{max} \{ 2\ell : \ell \in \{j_1, \dots j_r\} \} < \operatorname{min} \{ 2s-1 : s \in \{i_1, \dots i_k \} \} \]

\noindent that is, iff $\operatorname{max} \{ j_1, \dots j_r \} < \operatorname{min} \{ i_1, \dots i_k \}$, so we are done.
\end{proof}

\begin{rmk} \label{cop-picture}
The $\infty$-compatible order property decribes an interaction between two $P_\infty$-complete graphs,
i.e. consistent types. 
The hypotheses (1)-(2) of Claim \ref{comp-claim} are enough to allow a weak description of intervals. 
That is, we choose the sequences $\alpha_i$, $\beta_i$ to each describe a concentric sequence of intervals 
(each $\alpha_i$, $\beta_i$ corresponds to a set of matching parentheses)
along the sequence $\langle c_i \rangle$:
\begin{align*} \leftarrow[-[-[-[-&]-]-]-]-- \dots --(-(-(-(-)-)-)-)\rightarrow \\ 
\end{align*}
\noindent which we can interlace to obtain $\infty$-c.o.p. by judicious choice of indexing:
\[  \leftarrow 
 		\left[ 			  \ds 
 		\left[ \ds 	 
 		\left[ \ds 	 
 		\left[ \ds 	 
 		\left(  \ds 	 
									    	\right] \ds \left( \ds 
								 		 		\right] \ds \left( \ds 
										 	 	\right] \ds  \left(   \ds  
											 	\right]  \ds \right)  
 										 		\ds \right) \ds  
 										 		\right) \ds  
 										 		\right)  \rightarrow                \]
\noindent In this picture, the enumeration of the $\alpha$s $(~)$, would proceed from the outmost pair to the inmost
and the enumeration of the $\beta$s $[~]$ from inmost to outmost.
\end{rmk}

\begin{obs} \label{cop-p2}
Suppose that $\vp$ has the strict order property, i.e. there is an infinite sequence $\langle c_i : i<\omega \rangle$ on which 
$\exists x (\neg \vp(x;c_i) \land \vp(x;c_j)) \iff i<j$. Then $\neg \vp(x;y) \land \vp(x;z)$ has the $\infty$-compatible order property. 
\end{obs}

\begin{proof}
Writing $\rho(x;y,z) = \neg\vp(x;y) \land \vp(x;z)$,
\begin{itemize}
\item $\exists x \rho(x;c_i,c_j) \iff i<j$, by definition of strict order;
\item $\exists x (\rho(x;c_i,c_j) \land \rho(x;c_k,c_\ell)) \iff i,k < j,\ell$
\end{itemize}
and the characteristic sequence $P_\infty(\rho)$ has support 2. 
Apply Claim \ref{comp-claim}.
\end{proof}

\begin{expl}
The theory $T$ of the triangle-free random graph with edge relation $R$ has the $\infty$-c.o.p. 
Consider $\vp(x;y,z) = xRy \land xRz$. (The negative instances could be added but are not necessary.) Then: 

\begin{itemize}
\item $P_1((y,z)) \iff \neg yRz$.
\item $P_2((y,z), (y^\prime, z^\prime))$ iff $\{y, y^\prime, z, z^\prime\}$ is an empty graph. 
\item The sequence has support $2$, as the only problems come from a single new edge: 
$P_n ((y_1,z_1),\dots (y_n, z_n))$ iff  
\[ \exists x \left( \bigwedge_{i\leq n} xRy_i \land \bigwedge_{j \leq n} xRz_j  \right) 
~~\mbox{that is, if}~~ \bigcup_i y_i \cup \bigcup_j z_j  ~~\mbox{is a $P_2$-empty graph.}      \]
\end{itemize}

Let $\langle a_i, b_i : i<\omega \rangle$ be a sequence witnessing the \emph{$2$-empty} order property with respect to the 
edge relation $R$, say $a_i R b_j$ iff $j \leq i$. Then $\exists x (xRa_i \land xRb_j)$ iff $i<j$, i.e. $(a_i,b_j) \in P_1$ iff $i<j$.
Also, $\exists x (xRa_i \land xRb_j \land xRa_k \land xRb_\ell)$ if, in addition, $i, k < j,\ell$. Apply Claim \ref{comp-claim}. 
\end{expl}

Finally, we tie the compatible order property to $SOP_3$, a model-theoretic rigidity property.
$SOP_3$ will be discussed extensively in the next section, Definition \ref{sopn-defn}, Definition \ref{sop3-alt}.

\begin{lemma} \label{sop3-cop}
Suppose that $\theta(x;y)$ has $SOP_3$ in the sense of Definition \ref{sopn-defn}, so $\ell(x)=\ell(y)$.
Let $\vp_r = \vp$, $\psi_\ell = \psi$ be the formulas from Definition \ref{sop3-alt}. 
Then $\rho(x;y,z) := \vp_r(y,x) \land \psi_\ell(x,z)$ has the $\infty$-compatible order property
on some $A^\prime \subset P_1$. Moreover, we can choose $A^\prime$ so that the sequence restricted to $A^\prime$ has support 2.   
\end{lemma}

\begin{rmk}
This is an existential assertion, and it is straightforward to check that it remains true 
if we modify $\rho$ to include the corresponding negative instances. 
\end{rmk}

\begin{proof} (of Lemma)
Let $A := \langle a_i : i<\mathbb{Q} \rangle$ be an infinite indiscernible sequence from Definition \ref{sop3-alt}. Then 
\[ P_1((a_i, a_j)) \iff \exists x \left(\vp_r(a_i,x) \land \psi_\ell(x,a_j)\right) \iff i < j \] 
by the choice of $\vp, \psi$. More generally,
\[ P_n((a_{i_1}, a_{j_1}), \dots (a_{i_n}, a_{j_n})) \iff 
\exists x \left( \bigwedge_{t \leq n} \vp_r(x;a_{i_t}) \land \bigwedge_{t \leq n} \psi_\ell(x;a_{j_t})  \right)  \]

\noindent which, again applying Definition \ref{sop3-alt}, happens iff
$\operatorname{max} \{ i_1,\dots i_n \} < \operatorname{min} \{ j_1, \dots j_n \}$, a condition which has
support 2. We now apply Claim \ref{comp-claim} to obtain $A^\prime \subset A\times A$ witnessing the compatible order property. 
Note that while $\langle P_n \rangle$ need not depend on 2 elsewhere in $P_1$ (we know very little about $\rho$ off $A$), 
it does depend on 2 on elements from the sequence $A^\prime$. 
\end{proof}

\begin{obs}
Suppose $\theta(x;y)$ has the $\infty$-compatible order property. Then the formula $\vp(x;y,z) := \theta(x;y) \land \neg \theta(x;z)$ has $SOP_3$.
\end{obs}

\begin{proof}
Let $\langle a_i b_i : i<\omega \rangle$ be a sequence witnessing the $\infty$-compatible order property;
this will play the role of the sequence $\langle \overline{a}_i : i<\omega \rangle$ from Definition \ref{sop3-alt}. 
In the notation of that Definition, let $\vp(x;y,z) := \theta(x;y) \land \neg \theta(x;z)$ and
$\psi(x;y,z) := \theta(x;z)$. We check the conditions. 

(1) Clearly $\{ \vp(x;y,z), \psi(x;y,z) \}$ is inconsistent. 

(3) When $i > j$, $\{ \vp(x;a_i b_i), \psi(x;a_j b_j) \} = \{ \theta(x;a_i) \land \neg \theta(x; b_i), \theta(x;b_j) \}$
is inconsistent because $\neg P_2 (a_i, b_j)$. 

Finally, for $1 \leq j < \omega$ let 
$p_{j}(x) = \{ \theta(x;a_i) : 1\leq i\leq j \} \cup \{ \theta(x;b_\ell) : j < \ell <\omega \}$. 
The $\infty$-c.o.p. implies
$P_n(a_1,\dots a_j, b_{j+1}, \dots b_n)$ for all $n<\omega$, so $p_j$ is consistent. 
However, $i<j \implies \neg P_2(b_i, a_j)$ 
so $p_{j,n}(x) \vdash \neg \theta(x;b_i)$ for each $1\leq i \leq j$. Choosing $c_j \models p_j$ for each $j <\omega$ gives (2).
\end{proof}

\begin{rmk}
Applying Shelah's theorem that any theory with $SOP_3$ is maximal in the Keisler order \cite{Sh500}, \cite{ShUs},
we conclude that if $T$ contains a formula $\vp$ with the $\infty$-compatible order property, then $T$ is maximal in the Keisler order. 
For more on Keisler's order, see \cite{mm-article-1}.
\end{rmk}

\section{Calibrating randomness} \label{section:randomness}

In this final section,
we observe and explain a discrepancy between the model-theoretic notion of an infinite random $k$-partite graph 
and the finitary version given by Szemer\'edi regularity, showing essentially 
that a class of infinitary $k$-partite random graphs which do not admit reasonable finite approximations must have the 
strong order property $SOP_3$ (a model-theoretic notion of rigidity, Definition \ref{sopn-defn} below).

\subsection{A seeming paradox} 

\begin{obs} \label{tfrg}
Let $T$ be the theory of the triangle-free random graph, with edge relation $R$. 
Then it is consistent with $T$ that there exist disjoint infinite sets $X,Y,Z$ such that
each pair $(X,Y), (Y,Z), (X,Z)$ is a bipartite random graph. 
\end{obs} 

\begin{proof}
The construction has countably many stages. At stage 0, let $X_0 = \{ a \}, Y_0 = \{ b \}, 
Z_0 = \{ c \}$ where $a,b,c$ have no $R$-edges between them. At stage $i+1$, let $X_{i+1}$ be
$X_i$ along with $2^{|Y_i| + |Z_i|}$-many new elements:

\begin{enumerate}
\item for each subset $\tau \subset Y_i$, a new element $x_\tau$ such that for $y \in Y$,
$x_\tau R y$ $\iff y \in \tau$, however $\neg x_\tau R x$ for any $x$ previously added to $X_{i+1}$.  
  
\item for each subset $\nu \subset Z_i$, a new element $x_\nu$ such that for $z \in Z$,
$x_\nu R z$ $\iff z \in \nu$, with $x_\nu$ likewise $R$-free from previous elements of $X_{i+1}$.  
\end{enumerate} 

$Y_{i+1}, Z_{i+1}$ are defined symmetrically.
As we are working in the triangle-free random graph, in order that the the construction be able to continue,
it is enough that the sets $X_i, Y_i, Z_i$ are each empty graphs, i.e., at no point do we ask for a triangle. 

To finish, set $X = \bigcup_i X_i$, $Y= \bigcup_i Y_i$,
$Z = \bigcup_i Z_i$. Each pair is a bipartite random graph, as desired.
\end{proof}

But recall: 

\begin{thm-lit} (weak version of Key Lemma, Theorem \ref{key-lemma}) \label{key-lemma-1}
Fix $1> \delta > 0$ and a binary edge relation $R$.
Then there exist $\epsilon^\prime = \epsilon^\prime(\delta), N^\prime = N^\prime(\epsilon^\prime, \delta)$ such that: 
\emph{if}
$\epsilon < \epsilon^\prime$, $N > N^\prime$, 
$X, Y, Z$ are disjoint finite sets of size at least $N$, and each of the pairs 
$(X,Y), (Y,Z), (X,Z)$ is $\epsilon$-regular with density $\delta$, \emph{then}
there exist $x \in X, y \in Y, z \in Z$ so that $x,y,z$ is an $R$-triangle. 
\end{thm-lit}

Obviously, we cannot have an $R$-triangle in the triangle-free random graph. Nonetheless each of the pairs
$(X,Y)$ in Observation \ref{tfrg} manifestly has finite subgraphs of any attainable density. 

The difficulty comes when we try to choose finite subgraphs $X^\prime \subset X, Y^\prime \subset Y, Z^\prime \subset Z$
so that the densities of all three pairs are \emph{simultaneously} near the same $\delta > 0$. If $(X^\prime, Y^\prime)$
and $(Y^\prime, Z^\prime)$ are reasonably dense, $(X^\prime, Z^\prime)$ will be near $0$. Put otherwise, we may choose
elements of $X$ independently over $Y$, and independently over $Z$, but not both at the same time. 

The constructions below generalize this example, and give a way of measuring the ``depth'' of independence in a constellation
of sets $X_1, \dots X_n$, where any pair $(X_i, X_j)$ is a bipartite random graph. The example of the triangle-free random graph 
is paradigmatic: we shall see that a bound on the depth of independence will produce the $3$-strong order property
$SOP_3$.  

\subsection{Constellations of independence properties.}

\begin{defn}  
\label{constellations} Fix a formula $R(x;y)$.
\begin{enumerate}
\item Let $A, B$ be disjoint sets of $k$- and $n$-tuples respectively, where $k=\ell(x), n=\ell(y)$. 
Then $A$ is \emph{independent over $B$} with respect to $R$ just in case for any two finite disjoint $\eta, \nu \subset B$,
there exists $a \in A$ such that $b \in \eta \rightarrow R(a;b)$ and $b\in \nu \rightarrow \neg R(a;b)$.

\item Let $A_1,\dots A_k$ be disjoint sets (of $m$-tuples, where $m=\ell(x)=\ell(y)$). 
Then $A_1$ is \emph{independent over $A_2,\dots A_k$} with respect to $R$ just in case
$A_1$ is independent over $B := \bigcup_{2\leq i \leq k} A_i$ in the sense of (2).

\item $R(x;y)$ is a \emph{bipartite random graph} if there exist disjoint infinite sets $A,B$ such that 
$A$ and $B$ are each independent over the other wrt $R$.

\item $R(x;y)$ is \emph{$I^m_k$}, for some $2 \leq k \leq m$, if there exist disjoint infinite sets $\langle A_i : i<m \rangle$ such that
for any distinct $i_1,\dots i_k < \omega$, $A_{i_1}$ is independent over $\bigcup_{2 \leq j \leq k} A_{i_j}$ w.r.t. $R$. 
Note that $k$ refers to the \emph{depth} of the independence, and not the size of the finite disjoint $\eta, \nu$. 
\end{enumerate}
\end{defn}

\begin{obs}  
Let $R(x;y)$ be a symmetric formula. The following are equivalent.
\begin{enumerate}
\item $R$ is $I^\omega_\omega$.
\item There is an infinite subset of the monster model on which $R$ is a random graph. (Certainly this
need not be definable or interpretable in any way).
\end{enumerate}
\end{obs}

\begin{defn} \emph{(Shelah, \cite{Sh500}:Definition 2.5)} \label{sopn-defn}
For $n \geq 3$, the theory \emph{$T$ has $SOP_n$} 
if there is a formula $\vp(x;y)$, $\ell(x)=\ell(y)= k$,
$M \models T$ and a sequence $\langle a_i : i<\omega \rangle$ with
each $a_i \in M^k$ such that:

\begin{enumerate}
\item $M \models \vp(a_i, a_j) $ for $i<j<\omega$
\item $M \models \neg \exists x_1,\dots x_n 
(\bigwedge\{ \vp(x_m, x_k) : m<k<n ~\mbox{and}~ k=m+1~ \mbox{mod}~ n \})$
\end{enumerate}
\end{defn}

\begin{thm-lit} \emph{(Shelah, \cite{Sh500}: (1) is Claim 2.6, (2) is Theorem 2.9)} \label{sopn-theorems}
\begin{enumerate}
\item For a theory $T$, $SOP \implies SOP_{n+1} \implies SOP_n$, for $n \geq 3$ (not necessarily for the same formula).
\item If $T$ is a complete theory with $SOP_3$, then $T$ is maximal in the Keisler order. 
\end{enumerate}
\end{thm-lit}

We will derive $SOP_3$ from failures of randomness, using the following equivalent definition. Remember that, by convention, $a_i, x, \dots$ need not be singletons. 

\begin{defn} \label{sop3-alt} \emph{(\cite{ShUs}:Fact 1.3)} 
$T$ has $SOP_3$ iff there is an indiscernible sequence $\langle a_i : i<\omega \rangle$ and $\mathcal{L}$-formulas $\vp(x;y), \psi(x;y)$ such that:

\begin{enumerate}
\item $\{ \vp(x;y), \psi(x;y) \}$ is contradictory.
\item there exists a sequence of elements $\langle c_j : j<\omega \rangle$ such that
\begin{itemize}
\item $i\leq j$ $\implies$ $\vp(c_j; a_i)$
\item $i > j$ $\implies$ $\psi(c_j; a_i)$
\end{itemize}
\item if $i<j$, then $\{ \vp(x;a_j), \psi(x;a_i) \}$ is contradictory.
\end{enumerate}
\end{defn}

The idea of the construction (Theorem \ref{independence-at-n}) is contained in the following straightforward example. 

\begin{expl} \label{expl-rg}
Let $T$ be the triangle free random graph, with edge relation $R$. 
Then $R$ is $I^3_2$ but not $I^3_3$, and $T$ is $SOP_3$. 
\end{expl}

\begin{proof}
Let us prove the final clause (for the rest see Observation \ref{tfrg} and the discussion following). 

The theory by definition contains a forbidden configuration, a triangle. Suppose $A,B,C$ are disjoint infinite sets
witnessing $I^3_2$. Let us construct a sequence of triples $ S = \langle a_i, b_i, c_i : i<\omega \rangle$ such that,
for $i < \omega$,

\begin{itemize}
\item For all $j \leq i$, $b_i R a_j$. 
\item For all $j \leq i$, $c_i R b_j$.
\item For all $j \leq i$, $a_{i+1} R c_j$. 
\end{itemize}

\noindent Define a binary relation $<_\ell$ on triples by:
\[  (x,y,z) \leq_\ell (x^\prime, y^\prime, z^\prime) \iff \left( (xRy^\prime \land yRz^\prime \land zRx^\prime) \right) \]

\vspace{1mm}
\noindent While $<_\ell$ need not be a partial order on the model, it does linearly order the sequence $S$ by construction.
Looking towards Definition \ref{sop3-alt}, let us define two new formulas (the variables $t$ stand for triples):
\begin{itemize}
\item $\vp(t_0; t_1, t_2) = t_1 <_\ell t_2 <_\ell t_0$
\item $\psi(t_0; t_1, t_2) = t_0 <_\ell t_1 <_\ell t_2$ 
\end{itemize}

Let us check that these formulas give $SOP_3$. 
For condition (1), 
$\vp(t_0;t_1,t_2), \psi(t_0;t_1,t_2)$ means that 
$(x_0, y_0, z_0) <_\ell (x_1, y_1, z_1) <_\ell (x_2, y_2, z_2) <_\ell (x_0, y_0, z_0)$. 
Then $x_i R y_j$, $y_j R z_k$, $z_k R x_i$ which gives a triangle, contradiction.  

It is straightforward to satisfy (2) by compactness (e.g. by choosing $S$ codense in a larger indiscernible sequence).  

Finally, for condition (3), suppose $i<j$ but $\vp(t;\gamma_i), \psi(t;\gamma_j)$ is consistent, where $t=(x,y,z)$.
This means that $(x,y,z) <_\ell (a_i, b_i, c_i) <_\ell (a_j, b_j, c_j) <_\ell (x,y,z)$ (where the middle $<_\ell$
comes from the behavior of $<_\ell$ on the sequence $S$). 
As in condition (1), this gives a triangle, contradiction. 
\end{proof}

We can extend this idea to a much larger engine for producing the rigidity of $SOP_3$ from a forbidden configuration. 

\begin{theorem} \label{independence-at-n}
Suppose that for some $2 \leq n < \omega$, the formula $R$ of $T$ is $I^{n+1}_n$ but not $I^{n+1}_{n+1}$. 
Then $T$ is $SOP_{3}$. 
\end{theorem}

\begin{proof}  The construction is arranged into four stages.

\vspace{3mm}
\noindent\emph{Step 1: Finding a universally forbidden configuration $G$.}

\br
By hypothesis, $R$ is not $I^{n+1}_{n+1}$. This means that the infinitary type $p(X_0, \dots X_n)$,
which describes $n+1$ infinite sets $X_i$ which are $I^{n+1}_{n+1}$ in the sense of Definition \ref{constellations},
is not consistent.
Let $G$ be a finite inconsistent subset of height $h$, in the variables $V_G = \{ x^i_j : 1\leq i \leq h, 0\leq j\leq n \}$, 
and described by the edge map $E_G: \{ ((i,j),(i^\prime, j^\prime)) :  i,i^\prime \leq h, j \neq j^\prime \leq n \} \rightarrow \{ 0,1 \}$.
As the inconsistency of $p$ is a consequence of $T$, $G$ will be a universally forbidden configuration:

\begin{equation}
T \vdash \neg(\exists x^1_0,\dots x^h_n) \left( \bigwedge_{i,i^\prime \leq h, ~j \neq j^\prime \leq n} 
R(x^i_j, x^{i^\prime}_{j^\prime}) \iff E((i,j), ({i^\prime},{j^\prime}))=1 \right)
\end{equation}

\noindent Note that the configuration remains agnostic on edges between elements in the same column, in keeping with the
definition of $I^m_\ell$.

In what follows $G$ will appear as a template which we shall try to approximate using $I^{n+1}_n$. Here are the vertices
of $G$ arranged as they will be visually referenced (the edges are not drawn in):

\ssp
\begin{figure}[ht]
$\begin{array}{lllll}
x^h_0 & 	 & x^h_k & 	 & x^h_{n} \\
\vdots &   & \vdots   &      & \vdots  \\
x^\rho_0 &   & x^\rho_k  &      & x^\rho_{n}  \\
\vdots &   & \vdots   &      & \vdots  \\
x^1_0 & \dots	 & x^1_k & \dots		& x^1_{n}
\end{array}$

\caption{Vertices of the forbidden configuration $G$, arranged in columns. When comparing this configuration
to an array whose rows are indexed modulo $h$, the superscript of the top column becomes $0$.}
\end{figure}

\noindent\emph{Step 2: Building an array $A$ of approximations to $G$.}
\br

Let $A_0, \dots A_{n}$ be disjoint infinite sets witnessing $I^{n+1}_n$ for $R$. 
As in Example \ref{expl-rg}, we will use elements from these columns $A_i$ to build an array
$A = \langle a^\rho_i : 1 \leq \rho < \omega, 0 \leq i \leq n \rangle$. Fixing notation,

\begin{itemize}
\item $a^\rho_0,\dots a^\rho_n$ is called a \emph{row}. 
\item $\cols(i) = \{ j : j \neq i, i+1~ (\operatorname{mod} n+1) \}$ is the set of column indices
associated to the column index $i$.  
\item Define an ordering on pairs of indices ($\beta$ for ``before''): 
\begin{align*}
\beta((t^\prime, i^\prime),(t, i)) & \iff_{def} \\
\left( \vrt ( t^\prime < t \land i^\prime \in \cols(i) ) \right. & \lor \left. ( t^\prime = t \land i^\prime < i ) \vrt \right) \\
\end{align*}
\end{itemize}

\begin{claim} \label{array-claim} We may build the array $A$ to satisfy:
\begin{enumerate}
\item For all $\rho$, $a^\rho_k \in A_k$.
\item For any $\rho^\prime, \rho, k, k^\prime$ such that $\beta((\rho^\prime, k^\prime), (\rho, k))$,

\[  a^\rho_k ~R~ a^{\rho^\prime}_{k^\prime}   ~ \iff ~  E_G((r,k),({r^\prime},{k^\prime})) = 1       \]

\vspace{2mm}
\noindent where $r \equiv \rho$ (mod $h$), $r^\prime \equiv \rho^\prime$ (mod $h$). 
\end{enumerate}
\end{claim}

\begin{proof}
We choose elements in a helix ($a^1_0, a^1_1, \dots a^1_n, a^2_0, a^2_1, \dots$) so that
$\beta((\rho^\prime, k^\prime), (\rho, k))$ implies that $a^{\rho^\prime}_{k^\prime}$ is chosen before $a^\rho_k$.

When the time comes to choose $a^\rho_k$, we look for an element of $A_k$ which satisfies Condition (2) of the Claim, that is,
which, by Condition (1), realizes a given $R$-type over disjoint finite subsets of the columns $A_i$ ($i \in \cols(k)$). As $(A_0, \dots A_n)$
was chosen to be $I^{n+1}_n$ and $|\cols{k} | = n-1$, an appropriate $a^\rho_k$ exists.   
\end{proof}

\ssp
\begin{figure} 
\label{array-A}
\begin{small}
$\begin{array}{llllcllc}
& & &  & \vdots & &  &  \\
& & &  & \vdots & &  &  \\
& & & & & & & \\
&---&---&---&---&---&---&---\\

		&	a^{{\ell_n}h+h}_0 		& 		&\dots 	&   &{a^{{\ell_n}h+h}_k} 	& 	 	& \mathbf{a^{{\ell_n}h+h}_n} \\
		&	\vdots 		&  		& 			& 	& \vdots   				&     & \vdots  \\
B_{\ell_{n}} = \hspace{10mm} 	&	a^{{\ell_n}h+\rho}_0 	&  		& 			&  	&{a^{{\ell_n}h+\rho}_k}  &     & \mathbf{a^{{\ell_n}h+\rho}_{n}}  \\
		&	\vdots 		&  		& 			&  	&\vdots   				&     & \vdots  \\  
		&	a^{{\ell_n}h+1}_0 		& 		&	 			& 	& {a^{{\ell_n}h+1}_k} 		& 		& \mathbf{a^{{\ell_n}h+1}_n} \\

&---&---&---&---&---&---&---\\

& & & & && & \\
& & &  & \vdots & &  &  \\
& & & & & & & \\

&---&---&---&---&---&---&---\\

&a^{{\ell_k}h+h}_0 			& 	&\dots 	& & \mathbf{a^{{{\ell_k}h+h}}_k} 		& 	 			& a^{{\ell_k}h+h}_{n} \\
&\vdots 			&  	& 			& & \vdots   									&      		& \vdots  \\
B_{\ell_k}= &a^{{\ell_k}h+\rho}_0 		&  	& 			& & \mathbf{a^{{\ell_k}h+\rho}_k}  	&      		& a^{{\ell_k}h+\rho}_{n}  \\
&\vdots 			&  	& 			& & \vdots   									&      		& \vdots  \\
&a^{{\ell_k}h+1}_0 			& 	&\dots	& & \mathbf{a^{{\ell_k}h+1}_k} 			& \dots		& a^{{\ell_k}h+1}_{n} \\

&---&---&---&---&---&---&---\\

& & & & && & \\
& & &  & \vdots & &  &  \\
& & & & & & & \\

&---&---&---&---&---&---&---\\

&a^{2h}_0 		& \mathbf{a^{2h}_1}			& 		&	 	& a^{2h}_k 		& 	 		& a^{2h}_{n} \\
&\vdots 		& \vdots  							&   	&		& \vdots				&    		& \vdots  \\
B_{1} = &a^{h+\rho}_0 	&  \mathbf{a^{h+\rho}_1}& 		&  	& a^{h+\rho}_k  &      	& a^{h+\rho}_{n}  \\
&\vdots 		& \vdots   							&    	&		& \vdots  				&  			& \vdots  \\
&a^{h+1}_0 		& \mathbf{a^{h+1}_1}		& 		&		& a^{h+1}_k 		& \dots	& a^{h+1}_{n} \\

&---&---&---&---&---&---&---\\

&\mathbf{a^h_0} 							& 	&	&			 	& a^h_k 		& 	 		& a^h_{n} \\
&\mbox{\boldmath{$\vdots$}} 	&  	&	& 			& \vdots   	&      	& \vdots  \\
B_{0} = &\mathbf{a^\rho_0} 					&  	&	& 			& a^\rho_k  &      	& a^\rho_{n}  \\
&\mbox{\boldmath{$\vdots$}} 	&  	&	& 			& \vdots   	&    	  & \vdots  \\
&\mathbf{a^1_0} 							& 	&	&\dots	& a^1_k 		& \dots	& a^1_{n}
\end{array}$
\end{small}

\caption{Elements of the array $A$, arranged in blocks of $h$ rows. 
The boldface refers to Step 4 of the proof, when a proposed witness to $G$ is assembled
from the $i$th columns of blocks $B_i$ in a pseudo-$(n+1)$-loop.}
\end{figure}
\dsp

\br

\noindent\emph{Step 3: Defining the relation $<_\ell$, which has no pseudo-$(n+1)$-loops.}

\br
We now define a binary relation $<_\ell$ on $m$-tuples, where $m = h(n+1)$. Fix the enumeration
of these tuples to agree with the natural interpretation as blocks $B_\ell$ of $h$ consecutive rows
in the array $A$ (see Figure \ref{array-A}). That is, 
write the variables $Y := \langle y^t_i : 1 \leq t \leq h, 0 \leq i \leq n \rangle$, 
$Z := \langle z^{t^\prime}_{i^\prime} : 1 \leq t^\prime \leq h, 0 \leq i^\prime \leq n \rangle$. Define:  

\br
\begin{align*}
\mathbf{Y <_\ell Z} & ~{\iff_{(def)}}~~\\ 
\bigwedge_{1 \leq t^\prime, t \leq h, ~ 0\leq i, i^\prime \leq n} & \left(i^\prime \in \cols(i)\right) \implies \left( z^t_i ~R~ y^{t^\prime}_{i^\prime}  ~\iff ~ E_G((t,i),({t^\prime},{i^\prime})) = 1  \right) \\
\end{align*}

\br 

Let $B$ be a partition of the array $A$ into blocks $B_k$ ($k<\omega$) each consisting of $h$ consecutive rows,
so $B_k := \langle a^r_t : 0 \leq t \leq n, kh+1 \leq r \leq (kh)+h \rangle$, for each $k<\omega$ (see Figure \ref{array-A}).  
By Claim \ref{array-claim}, $i \lneq j \implies B_i <_\ell B_j$. 

\begin{defn}
A \emph{pseudo-$(n+1)$-loop} is a sequence $W_i$ $(0 \leq i \leq n)$ such that for some $m$, $1 \leq m < n$:
\begin{equation} \label{loop}
\left( \bigwedge_{(0 < j < i \leq n)} W_j <_\ell W_i \right)  ~\land~ \left( \bigwedge_{1 \leq j \leq m} W_0 <_\ell W_j \right)
\land \left( \bigwedge_{m < j \leq n} W_j <_\ell W_0 \right)
\end{equation}

\end{defn}

Suppose it were consistent with $T$ to have blocks of variables $W_0 \dots W_n$ which form a pseudo-$(n+1)$-loop.
Write $W_k(i) = \{ w^{hk+1}_i, \dots w^{hk+h}_i \}$ for the $i$th column of block $W_k$. Figure \ref{array-A} gives the picture,
where the elements $a$ are replaced by variables $w$ and the blocks $B_i$ become $W_i$.
Set $W_G = W_0(0) \cup \dots \cup W_n(n)$ ~(which can be visualized as the boldface columns in Figure \ref{array-A}). 

By definition of $<_\ell$, the pseudo-(n+1)-loop (\ref{loop}) implies that whenever
\[ \left( (~j \in \cols(i)) \land \left((0 < j < i \leq n) \lor (j=0 \land i \leq m) \lor (m < j \land i=0) \right) \right) \]
we will have:
\begin{align*}
\left(\forall~ w^t_k \in W(i),~ w^{t^\prime}_{k^\prime}\in W(j) \right) & \left( w^t_k~R~w^{t^\prime}_{k^\prime} \iff E_G((t,k),({t^\prime},{k^\prime})) = 1 \right) \\
\end{align*}

In other words, $<_\ell$ says that on certain pairs of elements in our proposed instance $W_G$ of $G$, namely those elements whose respective columns 
``fall into each other's scope'' as given by the $\cols$ operator, 
$W_G$ faithfully follows the template of $G$. 
It is easy to check that in a pseudo-$(n+1)$-loop every pair $j \neq i$ in $\{0, \dots n \}$ has this property. 
Thus pseudo-$(n+1)$-loops in $<_\ell$ are inconsistent with $T$.  

\br
\noindent\emph{Step 4: Obtaining $SOP_3$.}
\br

Step 3 showed that our array $A$ of approximations had a certain rigidity, which we can now identify as $SOP_3$.
Following Definition \ref{sop3-alt}, let us define 
$\vp_r(x;y_1,\dots y_n)$ and $\psi_\ell(x;y_1,\dots y_n)$,
where the the variables are blocks, and the subscripts ``$\ell$'' and ``r'' are visual aids: the element $x$ goes
to the left of the elements $y_i$ under $\psi$, and to their right under $\vp$. 

That is, we set:

\begin{itemize}
\item $\vp_r(x;y_1,\dots y_n) =$
\[ \bigwedge_{1 \leq i\neq j \leq n} y_i <_\ell y_j   ~ \land ~ \bigwedge_{1 \leq i \leq n} y_i <_\ell x \]
\item $\psi_\ell(x;y_1,\dots y_n) =$
\[ \bigwedge_{1 \leq i \leq n} x <_\ell y_i ~ \land ~ \bigwedge_{1 \leq i\neq j \leq n} y_i <_\ell y_j \]
\end{itemize} 

Now let us verify that the conditions of Definition \ref{sop3-alt} hold. Let $B$ be the sequence of blocks defined in Step 3, and
assume without loss of generality that $B = \langle B_k  : k <\omega \rangle$ is indiscernible and 
moreover is dense and codense in some indiscernible sequence $B^\prime$. Let $A = \langle A_i : i<\omega \rangle$
be an indiscernible sequence of $n$-tuples of elements of $B$. 

\begin{enumerate}
\item $\{ \vp_r(x;y_1,\dots y_n), \psi_\ell(x;y_1,\dots y_n) \}$ is contradictory because it gives rise to a 
pseudo-$(n+1)$-loop. 
\item By construction, for any $k < \omega$, the type
\[ \{ \psi_\ell(x;A_j) : j \leq k \} \cup \{ \vp_r(x;A_i) : k < i  \} \] is consistent, because $<_\ell$ linearly orders $B$,
thus also $B^\prime$. 
Choose the desired sequence of witnesses to be elements in the indiscernible sequence $B^\prime$ which are interleaved with $B$. 
\item Suppose we have $\{ \vp_r(x;A_j), \psi_\ell(x;A_i) \}$ for some $i<j$, or in other words:
\[ \{ \vp_r(x;B_{j_1},\dots B_{j_n}), \psi_\ell(x;B_{i_1},\dots B_{i_n})  \}  ~~ \mbox{where}~~ \{ i_1, \dots i_n\} < \{ j_1,\dots j_n \} \]
Then $x <_\ell B_{i_1} <_\ell \dots <_\ell B_{i_n} <_\ell B_{j_1} <_\ell \dots <_\ell B_{j_n} <_\ell x$ is a pseudo-$(2n+1)$-loop
(remember that $<_\ell$ holds between any increasing pair of elements of $B$ by construction). 
Thus a fortiori we have a pseudo-$(n+1)$-loop, contradicting the conclusion of Step 3.
\end{enumerate}

We have shown that the theory $T$ has $SOP_3$, so we finish. 
\end{proof}


\begin{thebibliography}{50}
 
\bibitem{es}
Elek and Szegedy, ``Limits of Hypergraphs, Removal and Regularity Lemmas.  A Non-standard Approach,'' (2007) arXiv:0705.2179. 

\bibitem{gowers}
Gowers, ``Hypergraph Regularity and the multidimensional Szemer\'edi Theorem.'' Ann. of Math. (2)  166  (2007),  no. 3, 897--946.

\bibitem{keisler} 
Keisler, ``Ultraproducts which are not saturated.'' Journal of Symbolic Logic, 32 (1967) 23--46.

\bibitem{rsurvey}
Koml\'os and Simonovits, ``Szemer\'edi's Regularity Lemma and its Applications in Graph Theory,''
Combinatorics, Paul Erd\"os is Eighty, vol. 2, Budapest (1996) 295--352.

\bibitem{mm-article-1} 
Malliaris, ``Realization of $\vp$-types and Keisler's order,'' Annals of Pure and Applied Logic 157 (2009) 220--224.

\bibitem{mm-article-2}
Malliaris, ``Persistence and NIP in the characteristic sequence,'' submitted (2009).

\bibitem{Sh:c} 
Shelah, \emph{Classification Theory and the number of non-isomorphic models}, revised edition. North-Holland, 1990. 

\bibitem{Sh500}  
Shelah, ``Toward classifying unstable theories,'' Annals of Pure and Applied Logic 80 (1996) 229--255. 

\bibitem{ShUs}
Shelah and Usvyatsov, ``More on ${\rm SOP}_1$ and ${\rm SOP}_2$,'' Annals of Pure and Applied Logic
155  (2008),  no. 1, 16--31.

\bibitem{sz}  
Szemer\'edi, ``On Sets of Integers Containing No $k$ Elements in Arithmetic Progression.'' Acta Arith. 27, 199-245, 1975a.

\bibitem{Usv}
Usvyatsov, ``On generically stable types in dependent theories.'' Journal of Symbolic Logic 74 (2009), 216-250.
\end{thebibliography}
\end{document}